\documentclass[12pt]{amsart}
\usepackage{latexsym,amsmath, amscd, amssymb, amsthm,  euscript, mathrsfs, hyperref}
\usepackage[dvipsnames]{xcolor}
\usepackage[all]{xy}
\usepackage{tikz}
\usetikzlibrary{arrows}
\usetikzlibrary{cd}
\usepackage{makecell}

\usepackage[cal=boondoxo, scr=boondoxo]{mathalfa}

\usepackage[totalwidth=480pt, totalheight=680pt]{geometry}

\newtheorem{thm}[subsection]{Theorem}
\newtheorem{thm/def}[subsection]{Theorem/Definition}

\newtheorem*{lem*}{Lemma}

\newtheorem{prop}[subsection]{Proposition}
\newtheorem*{prop*}{Proposition}

 \newtheorem{corollary}[subsection]{Corollary}
\newtheorem{lemma}[subsection]{Lemma}

\theoremstyle{definition}

\newtheorem{defn}[subsection]{Definition}

\theoremstyle{definition}

\theoremstyle{definition}

\newtheorem{rem}[subsection]{Remark}
\newtheorem*{rem*}{Remark}

\newtheorem{example}[subsection]{Example}
\numberwithin{equation}{subsection}

\theoremstyle{definition}
\newtheorem{definition}[subsection]{Definition}
\newtheorem{remark}[subsection]{Remark}

\newcommand{\et}{\text{\rm \'et}}
\newcommand{\mls }{\mathcal}
\newcommand{\Sp }{\text{\rm Spec}}

\title{Gerbes for trigonalizable group schemes}
\author{Noah Olander and Martin Olsson}
\date{}

\AtBeginDocument{%
   \def\MR#1{}
}

\begin{document}
\maketitle

\begin{abstract}
    We prove that smooth, separated Deligne--Mumford stacks in mixed characteristic with quasi-projective coarse moduli space are global quotient stacks and satisfy the resolution property. This builds on work of Kresch and Vistoli and of Bragg, Hall, and Matthur which proves the case when the stack is over a base field, as well as work of Gabber and de Jong which proves the same holds for a $\mu_n$-gerbe over a scheme with an ample line bundle.   The key technical input is a result that gerbes banded by so-called trigonalizable group schemes admit faithful vector bundles and are quotient stacks.
\end{abstract}

\section{Introduction}

Many moduli spaces and other algebraic stacks that occur in applications have stronger properties than being merely algebraic; rather, they can be presented as quotients of a scheme by the action of an algebraic group (they are \emph{quotient stacks}).  This property of being a quotient stack has a number of important consequences and enables one to use methods of equivariant geometry (for example, geometric invariant theory).  It is therefore an important question in the theory of algebraic stacks to understand intrinsic conditions on a stack which imply that it is a quotient stack.   With some mild assumptions, the property of being a quotient stack is equivalent to the existence of a certain vector bundle on the stack, which is also closely related to the so-called resolution property \cite{Totaro}.

An important classical application of this circle of ideas and questions is to the problem of comparing  the Brauer group $\text{Br}(S)$ of a scheme $S$, defined in terms of Azumaya algebras \cite{GroupedeBrauerI}, and the cohomological Brauer group $\text{Br}'(S)$, defined as $H^2(S_\et , \mathbf{G}_m)_{\text{tors}}$ (torsion in \'etale cohomology).  The basic problem, which has been resolved in a number of cases but remains open in general, is to understand when the natural map $\text{Br}(S)\rightarrow \text{Br}'(S)$ is an isomorphism.  A fruitful approach to this question is to reformulate the problem in terms of gerbes and their geometric structure.  Namely, a class in $\text{Br}'(S)$ can be represented by a $\mathbf{G}_m$-gerbe $\mls X\rightarrow S$, and the class is in the image of $\text{Br}(S)$ precisely when this gerbe $\mls X$ admits a nowhere zero $1$-twisted vector bundle (a vector bundle on which the stabilizer group schemes act by the standard character).  If $S$ furthermore admits an ample line bundle then this is also equivalent to $\mls X$ admitting a faithful vector bundle, and this is, in turn, also equivalent to $\mls X$ having the resolution property.    Key work establishing this connection is due to \cite{Caldararuthesis} and \cite{Lieblichthesis}.

From this perspective it is natural to consider the general question of when a gerbe $\mls X\rightarrow S$ banded by a group scheme $G$ has the resolution property.  That is the question studied in this article.  For $G = \mathbf{G}_m$ the question is part of the story of the Brauer group as noted above and is due to Gabber and de Jong \cite{dejong-gabber}, and in the case of unipotent group schemes 
%over a field -- they do a notion of relatively unipotent group scheme also
the problem was studied by Bragg, Hall, and Mathur \cite{bragg2025unipotentmorphisms} whose work informed many of the arguments here.

These works, however, do not answer the question even for some basic group schemes.   For example, let $S$ be a scheme with an ample line bundle and $n > 0$ an integer. If $\mathcal{X} \to S$ is a gerbe for the group $\mathbf{Z}/n$, then $\mathcal{X}$ is known to have a faithful vector bundle when (a) $n \in \mathcal{O}(S)^\times$ by Gabber's result \cite{dejong-gabber}, and (b) when $S$ is a scheme over a field by an extension of Gabber's Theorem due to Bragg, Hall, and Mathur \cite{bragg2025unipotentmorphisms}. The results of this article settle, in particular, this case in general.

\subsection{Statements of main results}
%For any integer $N>0$ let $D_N, T_N, U_N \subset GL_N$ denote the subgroup schemes of diagonal matrices, upper triangular matrices, and upper triangular matrices with $1$'s on the diagonal, respectively. These fit into a short exact sequence
%\begin{equation}
%\label{equn-diagbyunip}
%1 \to U_N \to T_N \to D_N \to 1
%\end{equation}
%of group schemes, where $T_N \to D_N \cong \mathbf{G}_m^N$ takes a matrix to its diagonal entries.

\begin{definition}[{\cite[4.2.3.4]{DemazureGabriel}}] Let $S$ be a scheme.  A %finitely presented, flat, 
relatively affine group scheme $G/S$ is \emph{trigonalizable} if it admits a faithful $G$-equivariant vector bundle $\mls  E$ which has a $G$-stable filtration 
$$
0 = \mls  E^n \subset \mls  E^{n-1}\subset \cdots \subset \mls  E^0 = \mls  E
$$
with each successive quotient $\mls  E^i/\mls  E^{i+1}$  a line bundle.
We say that a relatively affine group scheme $G/S$ is \emph{locally trigonalizable} if there exists a finite locally free surjection $S'\rightarrow S$ such that $G\times _SS'$ is trigonalizable.
\end{definition}

Given a group scheme $G/S$ and a filtered $G$-equivariant vector bundle $\mls E^\bullet$ as in the definition we get an induced morphism of group schemes
\begin{equation}
\label{equn-maptotorus}
\Psi^{\mathcal{E}^\bullet} :G\rightarrow \prod _i\text{Aut}(\mls E^i/\mls E^{i+1})%=:T
\simeq \mathbf{G}_{m, S}^n.
\end{equation}

\begin{definition}
Given a gerbe $\mathcal{X} \to S$ banded by a locally trigonalizable group scheme $G$, we say $\mathcal{X}$ is \emph{torsion} if 
for some choice of finite locally free surjective morphism $S'\to S$ and faithful $G_{S'}$-equivariant vector bundle $\mathcal{E}$ on $S'$ with complete $G_{S'}$-stable flag $\mathcal{E}^\bullet$, the class of the pushforward gerbe $\Psi^{\mathcal{E}^{\bullet}}_*\mathcal{X}_{S'}$ in $H^2(S'_{\acute{e}t}, \mathbf{G}_{m, S'}^n)$ is torsion.
\end{definition}

\begin{remark}
    In general, the condition depends on the choice of $G_{S'}$-stable filtration $\mathcal{E}^\bullet$. However, if $G \to S$ is a finite locally free group scheme which is locally trigonalizable, then every gerbe $\mathcal{X} \to S$ banded by $G$ is torsion. Also, if $G \to S$ is a locally trigonalizable group scheme and the finite locally free cover $S' \to S$ in the definition of locally trigonalizable can be chosen with $S'$ regular, then any gerbe $\mathcal{X} \to S$ banded by $G$ is torsion; see \ref{subsection-remarks}.
\end{remark}

\begin{thm}
\label{thm-faithfulbundleexists}
    Let $S$ be a scheme with an ample line bundle. Let $G \to S$ be a finitely presented, flat locally trigonalizable group scheme.  If $\mathcal{X} \to S$ is a gerbe banded by $G$ which is torsion, then $\mathcal{X}$ admits a faithful vector bundle and is a quotient stack.  If furthermore $G \to S$ is finite locally free, then $\mathcal{X}$ has the resolution property.
\end{thm}

\begin{remark} The key point here is that $\mls X$ admits a faithful vector bundle.  The statements about being a quotient stack and having the resolution property follow from this in a standard manner, see Section \ref{subsection-quotientstacks}.
\end{remark}

%\begin{remark} The condition that $\mls X_T$ is torsion holds automatically if $G$ if a finite locally free group scheme; see ??.
%\end{remark}

With this result the natural next question is which group schemes are locally trigonalizable. The case of group schemes over a field is well-understood using classical results about algebraic groups. 

\begin{thm}[Lemma \ref{lemma-liekolchin}] Let $S = \Sp (K)$ be the spectrum of a field and let $G/S$ be a finite type affine group scheme.  Then $G$ is locally trigonalizable if it is of one of the following types:

(a) Commutative;

%(b) Local; -- is this true?

(b) Smooth, solvable, and connected.
\end{thm}

Over general base schemes the main result is the following, discussed in Section \ref{S:section4}.

\begin{thm}
\label{thm-dedekindschemes}
    Let $S$ be an integral regular Nagata scheme of dimension $\leq 1$ with function field $K$.  If $G/S$ is an affine flat group scheme of finite type whose generic fiber $G_K$ is locally trigonalizable then $G$ is also locally trigonalizable.
\end{thm}

As an application of our main results, we prove that many smooth Deligne--Mumford stacks over bases are quotient stacks.  The following result is shown in Section \ref{S:section5}.

\begin{thm}
\label{thm-smoothDM}
    Let $S$ be a normal Noetherian scheme. Let $\mathcal{X}$ be a  separated Deligne--Mumford stack over $S$ which is smooth over $S$ and  whose coarse space admits an ample line bundle. Then $\mathcal{X}$ admits a faithful vector bundle, satisfies the resolution property, and is a quotient stack. 
\end{thm}

\begin{rem}
 This Theorem follows from Theorems \ref{thm-faithfulbundleexists} and \ref{thm-dedekindschemes} via standard techniques introduced in \cite{kreschvistoli}, where Theorem \ref{thm-smoothDM} is proven when $S$ is the spectrum of a field and $\mathcal{X}$ is generically tame. The hypothesis that $\mathcal{X}$ be generically tame was removed  in \cite{bragg2025unipotentmorphisms}, but not the hypothesis that $S$ is the spectrum of a field. 
\end{rem}
 \begin{rem} After circulating a preliminary draft of this article, the authors learned that variants of theorems \ref{thm-faithfulbundleexists} and \ref{thm-dedekindschemes} have been obtained independently by Česnavičius, Hall, and Mathur (unpublished).
 \end{rem}

 \medskip\noindent\textbf{Acknowledgements.} NO partially supported by the National Science Foundation under 
DMS-2402087, and MO  was partially supported by NSF FRG grant DMS-2151946 and the Simons Collaboration on Perfection in Algebra, Geometry, and Topology.

\section{Preliminaries}

\subsection{Algebraic stacks}

We follow the conventions on algebraic stacks and their quasi-coherent sheaves used in the Stacks Project \cite{stacks-project}. By algebraic stack we mean algebraic stack over $\operatorname{Spec}(\mathbf{Z})$ and by a vector bundle on an algebraic stack $\mathcal{X}$ we mean a locally free  $\mathcal{O}_{\mathcal{X}}$-module of finite rank.

Let $S$ be a scheme and let $G \to S$ be a finitely presented flat group scheme. A $G$-gerbe means an algebraic stack $\mathcal{X} \to S$ such that there is an fppf covering $\{U_i \to U\}_i$ with $\mathcal{X}_{U_i} \cong BG_{U_i}$ as algebraic stacks over $U_i$ for each $i$. A gerbe banded by $G$ is a $G$-gerbe $\mathcal{X} \to S$ together with an isomorphism between the band of $\mathcal{X}$ and the band associated to $G$ in the category of bands on $(Sch/S)_{fppf}$, see \cite[IV.1.1]{Giraud}. If $G$ is commutative then this is the same data as an isomorphism of sheaves of groups on $\mathcal{X}_{fppf}$ from the inertia stack of $\mathcal{X}$ to $G_{\mathcal{X}}$. 

\subsection{Quotient stacks and the resolution property}
\label{subsection-quotientstacks}

Let $\mathcal{X}$ be an algebraic stack. Recall that $\mathcal{X}$ is said to \emph{have the resolution property} if every quasi-coherent sheaf $\mathcal{F}$ of finite type on $\mathcal{X}$ is a quotient of a vector bundle. A vector bundle $\mathcal{E}$ on $\mathcal{X}$ is called \emph{faithful} if for every geometric point $\bar{x} : \operatorname{Spec}(\bar{k}) \to \mathcal{X}$, the associated representation of the stabilizer group scheme $\operatorname{Aut}_{\bar{k}}(\bar{x}) \to GL({\mathcal{E}(\bar{x})})$ is a monomorphism. The following is well-known, see \cite[Lemma 2.12]{ehkv} where it is proven in a Noetherian setting.

\begin{lemma}
\label{lemma-faithfulbundlemeaning}
    Let $\mathcal{X}$ be an algebraic stack and let $\mathcal{E}$ be a vector bundle on $\mathcal{X}$ of constant rank $n \geq 0$. Then $\mathcal{E}$ is faithful if and only if the total space of the associated $GL_n$-bundle
    $$
    Isom_{\mathcal{X}}(\mathcal{O}_{\mathcal{X}}^{\oplus n}, \mathcal{E}) \to \mathcal{X}
    $$
    is an algebraic space. In particular, a quasi-compact algebraic stack admits a faithful vector bundle if and only if it is a quotient $[Y/GL_N]$ with $Y$ an algebraic space and $N\geq 0$ an integer. 
\end{lemma}

\begin{proof}
   An algebraic stack is an algebraic space if and only if its stabilizers at geometric points are trivial \cite[8.1.1]{MR1771927}. The stabilizer group scheme of $Isom_{\mathcal{X}}(\mathcal{O}_{\mathcal{X}}^{\oplus n}, \mathcal{E})$ at a geometric point lying over the geometric point $\bar{x} : \operatorname{Spec}(\bar{k}) \to \mathcal{X}$ of $\mathcal{X}$ is isomorphic to the kernel of the associated homomorphism $\operatorname{Aut}_{\bar{k}}(\bar{x}) \to GL({\mathcal{E}(\bar{x})})$ which proves the first statement. For the second, if $\mathcal{X}$ is a quasi-compact algebraic stack with a faithful vector bundle, then it admits a faithful vector bundle $\mathcal{E}$ of constant rank $n$ (add trivial factors over finitely many closed and open substacks), and then $\mathcal{X} = [Y/GL_n]$ where $Y = Isom_{\mathcal{X}}(\mathcal{O}_{\mathcal{X}}^{\oplus n}, \mathcal{E})$ is an algebraic space. Conversely, if $\mls X= [Y/GL_N]$ then pulling back the vector bundle on $BGL_N$ corresponding to the standard representation we obtain a faithful vector bundle on $\mls X$.
\end{proof}

\begin{lemma}
\label{lemma-resforcover}
    Let $f : \mathcal{Y} \to \mathcal{X}$ be a finite locally free surjective morphism of algebraic stacks.
    \begin{enumerate}
        \item If $\mathcal{E}$ is a nowhere zero faithful vector bundle on $\mathcal{Y}$ then $f_*\mathcal{E}$ is a faithful vector bundle on $\mathcal{X}$.
        \item If $\mathcal{Y}$ satisfies the resolution property then so does $\mathcal{X}$.
    \end{enumerate}
\end{lemma}

\begin{proof}
    For (1), $f_* \mathcal{E}$ is a vector bundle since $f$ is finite locally free. Let $\bar{x} : \operatorname{Spec}(\bar{k}) \to \mathcal{X}$ be a geometric point and let $G_{\bar x}$ be the stabilizer group scheme.  Let $A$ denote the $\bar k$-algebra with $G_{\bar x}$-action $\bar x^*f_*\mls O_{\mls Y}$ so that $\mls Y_{\bar x}:= \mls Y\times _{\mls X}\bar x\simeq \Sp (A)$.  The fiber $(f_*\mls E)(\bar x)$ also has a $G_{\bar x}$-action, which we need to show is faithful, which is compatible with the natural faithful $A$-module structure in the sense that for every $\bar k$-algebra $R$ the diagram
    $$
    \xymatrix{
    G_{\bar x}(R)\times (A\otimes _{\bar k}R)\ar[d]_-{(g, a)\mapsto g(a)}\ar@{^{(}->}[r]& G_{\bar x}(R)\times \text{End}_R((f_*\mls E)(\bar x)\otimes _{\bar k}R)\ar[d]^-c\\
    A\otimes _{\bar k}R\ar@{^{(}->}[r]& \text{End}_R((f_*\mls E)(\bar x)\otimes _{\bar k}R)}
    $$
    commutes.  Here $c$ is the map induced by conjugation by the $G_{\bar x}(R)$-action on $(f_*\mls E)(\bar x)\otimes _{\bar k}R$. From the commutativity of this diagram we find that the kernel $\Gamma _{\bar x}$ of the map $G_{\bar x}\rightarrow GL((f_*\mls E)(\bar x))$ acts trivially on $A$, and therefore is contained in $G_{\bar y}\subset G_{\bar x}$ for any lift 
     $\bar{y} : \operatorname{Spec}(\bar{k}) \to \mathcal{Y}$ of $\bar x$. There is a morphism of $\bar{k}$-vector spaces
    \begin{equation}
    \label{equn-pushforwardonfibres}
     f_*\mathcal{E}(\bar{x}) = \bar{y}^*f^*f_*\mathcal{E} \to \bar{y}^*\mathcal{E} = \mathcal{E}({\bar{y}})
    \end{equation}
    which is surjective since $f$ is affine (so $f^*f_*\mathcal{E} \to \mathcal{E}$ is surjective). This is in fact a morphism of 
    representations of $G_{\bar{y}}$. Since (\ref{equn-pushforwardonfibres}) is surjective and $\mathcal{E}(\bar{y})$ is faithful, we see that 
    $(f_*\mathcal{E})(\bar{x})$ is faithful as a representation of $G_{\bar{y}}$. Then since $\Gamma _{\bar x} \subset G_{\bar{y}}$  acts trivially on $(f_*\mls E)(\bar x)$ we conclude that $\Gamma_{\bar{x}}$ is trivial, as needed.

    For (2), denote $f^!$ the right adjoint of $f_* : \operatorname{QCoh}(\mathcal{Y}) \to \operatorname{QCoh}(\mathcal{X})$ 
    and recall that $f_*f^!\mathcal{F} = \mathcal{H}om_{\mathcal{O}_{\mathcal{X}}}(f_*\mathcal{O}_{\mathcal{Y}}, \mathcal{F})$ 
    and this has the $f_*\mathcal{O}_{\mathcal{Y}}$-module structure defined by precomposition with multiplication by sections of $f_*\mathcal{O}_{\mathcal{Y}}$.  In fact, we don't need the theory of $f^!$ here, but rather can simply define $f^!\mls F$ for a quasi-coherent sheaf $\mls F$ on $\mls X$ to be the quasi-coherent sheaf  on $\mls Y$ corresponding to the quasi-coherent $f_*\mls O_{\mls Y}$-module on $\mls X$ given by $\mls Hom _{\mls O_{\mls X}}(f_*\mls O_{\mls Y}, \mls F)$.

    Now let $\mathcal{F}$ be a quasi-coherent sheaf of finite type on $\mathcal{X}$. Then $f^!\mathcal{F}$ is a quasi-coherent sheaf of finite type on $\mathcal{Y}$ and $\mathcal{Y}$ has the resolution property by assumption. Hence there is a surjection $\varphi : \mathcal{E} \to f^!\mathcal{F}$ with $\mathcal{E}$ a vector bundle on $\mathcal{Y}$. If $\psi: f_*\mathcal{E} \to \mathcal{F}$ corresponds to $\varphi$ under adjunction, then $\psi$ is a surjection from a vector bundle. To check surjectivity, note that $\psi$ is the composition 
    $$
    f_*\mathcal{E} \xrightarrow{f_*\varphi}f_*f^!\mathcal{F} \xrightarrow{\epsilon_{\mathcal{F}}} \mathcal{F}.
    $$
    Now $f_*\varphi$ is surjective since $f$ is finite and the co-unit of adjunction $\epsilon_{\mathcal{F}}$ is the map
    $$
    \mathcal{H}om_{\mathcal{O}_{\mathcal{X}}}(f_*\mathcal{O}_{\mathcal{Y}}, \mathcal{F}) \to \mathcal{H}om_{\mathcal{O}_{\mathcal{X}}}(\mathcal{O}_{\mathcal{X}}, \mathcal{F}) = \mathcal{F}
    $$
    defined by precomposition with $\mathcal{O}_{\mathcal{X}} \to f_*\mathcal{O}_{\mathcal{Y}}$. But since $f$ is finite locally free,  $\mathcal{O}_{\mathcal{X}} \to f_*\mathcal{O}_{\mathcal{Y}}$ is an injection which splits fppf locally on $\mathcal{X}$, hence $\epsilon_\mathcal{F}$ is surjective and we are done.
\end{proof}

Recall the following theorem of Gross, originally due to Totaro in the case $\mathcal{X}$ is Noetherian and normal.

\begin{thm}[\cite{Gross}, \cite{Totaro}]
\label{thm-totaro}
Let $\mathcal{X}$ be a quasi-compact and quasi-separated algebraic stack. Assume $\mathcal{X}$ has affine stabilizer group schemes at closed points. The following are equivalent:
\begin{enumerate}
    \item $\mathcal{X}$ has the resolution property.
    \item There exist a quasi-affine scheme $U$, an integer $n \geq 0$, and an action of $GL_n$ on $U$, such that $\mathcal{X} \cong [U/GL_n]$. 
\end{enumerate}
If the equivalent conditions hold, then $\mathcal{X}$ has affine diagonal.
\end{thm}

For stacks with finite diagonal we have the following result of Deshmukh, Hogadi, and Mathur which generalizes an earlier result of Kresch and Vistoli over a field.

\begin{thm}[\cite{1resolutionprop}, \cite{kreschvistoli}]
\label{thm-finiteflatcover}
    Let $\mathcal{X}$ be a quasi-compact algebraic stack with finite diagonal (as a stack over $\operatorname{Spec}(\mathbf{Z})$). Let $X$ be the coarse moduli space of $\mathcal{X}$ (which exists in this generality by \cite{MR3084720}). Assume $X$ admits an ample line bundle. Then the following are equivalent:
    \begin{enumerate}
        \item $\mathcal{X}$ admits a faithful vector bundle.
        \item There exists a finite locally free surjective morphism $Y \to \mathcal{X}$ where $Y$ is a scheme with an ample line bundle.
        \item $\mathcal{X}$ has the resolution property.
    \end{enumerate}
\end{thm}

\begin{proof}
    The difficult direction is (1) $\implies$ (2) which is proved in \cite[Corollary 4.5]{1resolutionprop}. Note that the statement there is that there exists a finite flat cover $f : Y \to \mathcal{X}$ where $Y$ is a scheme with an ample line bundle, but the morphism they construct in the proof is finite locally free. 

    (2) $\implies$ (3) follows from Lemma \ref{lemma-resforcover} and the fact that a scheme with an ample line bundle has the resolution property \cite[II, 2.2.3 and 2.2.4]{SGA6}.

    Finally, (3) $\implies$ (1) follows from Theorem \ref{thm-totaro} since we can write $\mathcal{X} \cong [U/GL_n]$ with $U$ a quasi-affine scheme. Then if $\mathcal{E}$ is the vector bundle on $\mathcal{X}$ corresponding to the universal $GL_n$-bundle $U \to [U/GL_n]$, then $\mathcal{E}$ is faithful by Lemma \ref{lemma-faithfulbundlemeaning}.
\end{proof}

This leads to the following notion for group schemes, which captures the key question considered in this paper.

\begin{defn} Let $S$ be a scheme.  A finite locally free group scheme $G\rightarrow S$ \emph{admits resolutions} if for every $S$-scheme $T$ with an ample line bundle and every gerbe $\mls X\rightarrow T$ banded by $G_T$, the stack $\mls X$ has the resolution property.
\end{defn}

\begin{corollary}
\label{corollary-extensionsofgroups}
    Let $S$ be a scheme. 
    The collection of finite locally free group schemes $G \to S$ which admit resolutions is closed under extensions.
\end{corollary}

\begin{proof}
    Let $1 \to G' \to G \to G'' \to 1$ be an extension of finite locally free group schemes over $S$ and assume $G', G''$ satisfy the property. Let $T$ be a scheme with ample line bundle and $\mathcal{X} \to T$ a gerbe banded by $G_T$. Let $\mathcal{X} \to \mathcal{Y}$ be the pushout along $G_T \to G_T''$. Then $\mathcal{Y}$ is banded by $G_T''$ and $\mathcal{X} \to \mathcal{Y}$ is a gerbe banded by $G_T'$. Since $G''$ satisfies the property, the stack $\mathcal{Y}$ satisfies the resolution property. By Theorem \ref{thm-finiteflatcover}, there is a finite locally free surjective morphism $Y \to \mathcal{Y}$ where $Y$ is a scheme with an ample line bundle. Then $\mathcal{X} \times _\mathcal{Y} Y \to Y$ is a $G_Y'$ gerbe and $G'$ satisfies the property so $\mathcal{X} \times _\mathcal{Y} Y$ has the resolution property. Then $\mathcal{X} \times _{\mathcal{Y}} Y \to \mathcal{X}$ is finite locally free surjective and the source satisfies the resolution property hence so does the target by Lemma \ref{lemma-resforcover}. 
\end{proof}

We can combine these results with Gabber's Theorem and the results of \cite{bragg2025unipotentmorphisms} to obtain the following application. This is independent from the main results of this paper but may nevertheless be of interest.

\begin{corollary}
    If $G$ is a finite group scheme over a field $k$ then $G$ admits resolutions. That is, if $S$ is a $k$-scheme with an ample line bundle and $\mathcal{X} \to S$ is a gerbe banded by $G$ then $\mathcal{X}$ admits a faithful vector bundle, is a quotient stack, and satisfies the resolution property.
\end{corollary}

\begin{proof}
    Let $G^0 \lhd G$ be the connected component of the identity. Then it suffices by Corollary \ref{corollary-extensionsofgroups} to prove the result for the connected group scheme $G^0$ and the \'etale group scheme $G/G^0$. The \'etale case follows from \cite[Theorem 6.3]{bragg2025unipotentmorphisms}. Therefore, we may assume $G = G^0$ is connected. If $\operatorname{char}(k) = 0$ then $G = 1$ and there is nothing to show. Assume $\operatorname{char}(k) = p > 0$. If $k = \bar{k}$, then $G$ has a sequence of subgroups
    $$
    1  = G_0 \lhd G_1 \lhd \cdots \lhd G_{n} = G
    $$
    with $G_{i+1}/G_i$ isomorphic to either $\alpha_p$ or $\mu_p$. In general, there is a finite extension $k'/k$ such that $G_{k'}$ admits such a sequence. If we prove the result for $\mathcal{X}_{k'}$ then it holds for $\mathcal{X}$ as well by Lemma \ref{lemma-resforcover}. Therefore we may assume $G$ admits such a sequence. But then by Corollary \ref{corollary-extensionsofgroups}, it suffices to prove the result when $G = \mu_p$ or $\alpha_p$. The first case follows from Gabber's Theorem \cite{dejong-gabber} and the second from \cite[Remark 6.1]{bragg2025unipotentmorphisms}.
\end{proof}

\section{Proof of Theorem \ref{thm-faithfulbundleexists}}

\subsection{Twisted sheaves}
\label{subsection-twistedsheaves}

Let $S$ be a scheme and $G/S$ a finitely presented flat group scheme. Let $\psi : G \to \mathbf{G}_{m, S}$ be a character. Let $\mathcal{X} \to S$ be a gerbe banded by $G$. Then Giraud constructs a morphism
$$
f_{\psi} : \mathcal{X} \to \mathcal{X}_{\psi}
$$
of gerbes over $S$ with $\mathcal{X}_{\psi}$ banded by $\mathbf{G}_{m, S}$ and such that on bands, $f_{\psi}$ is the morphism induced by $\psi : G \to \mathbf{G}_{m, S}$, see \cite[IV.3.1]{Giraud}. We call the gerbe $\mathcal{X}_{\psi}$ the pushforward of the gerbe $\mathcal{X}$ along $\psi$.

Recall that if $\mathcal{Y} \to S$ is a gerbe banded by $\mathbf{G}_{m, S}$, then a sheaf of $\mathcal{O}_{\mathcal{Y}}$-modules is \emph{1-twisted} if, under the identification of $\mathbf{G}_{m, \mathcal{Y}}$ with the inertia stack of $\mathcal{Y}$, $\mathbf{G}_{m, \mathcal{Y}}$ acts on $\mathcal{F}$ via the identity character.

\begin{definition}
    A sheaf $\mathcal{F}$ of $\mathcal{O}_{\mathcal{X}}$-modules is called \emph{$\psi$-twisted} if it is a pullback of a 1-twisted sheaf along $f_\psi : \mathcal{X} \to \mathcal{X}_{\psi}$.
\end{definition}

\begin{example}
    If $\mathcal{X} = BG_S$ is the trivial gerbe and $\mathcal{L}$ is the line bundle on $BG_{S}$ corresponding to the $G$-equivariant sheaf $\mathcal{O}_S$ on which $G$ acts via $\psi$, then $\mathcal{L}$ is $\psi$-twisted: It is the pullback of the tautological line bundle under $f_\psi : BG_S \to B\mathbf{G}_{m, S}$. 
\end{example}

\begin{lemma}
\label{lemma-twistedlocalstr}
    If $\mathcal{E}$ is a $\psi$-twisted vector bundle on $\mathcal{X}$, then there is an fppf covering $\{U_i \to S\}_i$ such that (1) $\mathcal{X}_{U_i} \cong BG_{U_i}$ and (2) $\mathcal{E}_{U_i} \cong \mathcal{L}_{U_i}^{\oplus r_i}$, where $\mathcal{L}$ is the line bundle on $BG_S$ from the example above. 
\end{lemma}

\begin{proof}
    The question is fppf local on $S$. We may certainly find an fppf covering such that (1) holds. We then may replace $\mathcal{X}$ by $\mathcal{X}_{U_i}$ and therefore assume $\mathcal{X} = BG_S$. Then by assumption, $\mathcal{E}$ is the pullback of a one-twisted vector bundle on $B\mathbf{G}_{m, S}$ (note -- $f_\psi$ is surjective and smooth so if $f_\psi^*\mathcal{H}$ is a vector bundle, then so is $\mathcal{H}$). These all have the form $\mathcal{F}\otimes \mathcal{U}$ where $\mathcal{F}$ is the pullback of a vector bundle on $S$ and $\mathcal{U}$ is the tautological line-bundle on $B\mathbf{G}_{m, S}$. Working Zariski locally on $S$, we may therefore assume $\mathcal{F} = \mathcal{O}^{\oplus r}$ so $\mathcal{E}$ is the pullback of $\mathcal{U}^{\oplus r}$, which is $\mathcal{L}^{\oplus r}$. 
\end{proof}

\subsection{Lemma on extensions}

The following arguments occur in \cite{bragg2025unipotentmorphisms}, where they play a pivotal role in the proofs of the main results. We repeat them here because they are equally important in our proof of Theorem \ref{thm-faithfulbundleexists}. The key is the following special case of Sch\"appi's Theorem. 

\begin{thm}{\cite[Theorem 1.3.1]{schäppi2012characterizationcategoriescoherentsheaves}}
    Let $S$ be a scheme with an ample line bundle. Let $g: U \to S$ be a faithfully flat morphism from an affine scheme. Then $g_*\mathcal{O}_U = \operatorname{colim}_i \mathcal{P}_i$ is a filtered colimit of vector bundles on $S$. %Furthermore, for $i  \gg 0$, the morphism $g^*\mathcal{P}_i \to \mathcal{O}_U$ corresponding to $\mathcal{P}_i \to g_* \mathcal{O}_U$ is a split surjection, and for $j \geq i$ we have compatible splittings
   % $$
   % g^*\mathcal{P}_j \cong \mathcal{O}_{U} \oplus \mathcal{Q}_j
   % $$
   % where $\mathcal{Q}_j$ is a system of vector bundles.% with colimit zero. 
\end{thm}

A proof of the special case considered above can also be found in \cite[Theorem 3.1]{bragg2025unipotentmorphisms}.

%\begin{proof}
%The first part is a special case of Sch\"appi's Theorem \cite[Theorem 1.3.1]{schäppi2012characterizationcategoriescoherentsheaves}, see also \cite[Theorem 3.1]{bragg2025unipotentmorphisms}. For the second part, since $S$ has an ample line bundle, it is quasi-compact and separated, so 
%$$
%H^0(S, g_* \mathcal{O}_U) = \operatorname{colim}_i H^0(S, \mathcal{P}_i),
%$$
%so the map $\mathcal{O}_S \to g_*\mathcal{O}_U$ factors through $\mathcal{P}_i$ for $i \gg 0$. Then adjunction gives morphisms $\mathcal{O}_U \to g^*\mathcal{P}_i \to \mathcal{O}_U$ which compose to the identity. The splittings can be chosen compatibly because if $s \in \Gamma(U, g^*\mathcal{P}_i)$ maps to $1 \in \Gamma(U, \mathcal{O}_U)$ then so does the image of $s$ in $\Gamma(U, g^*\mathcal{P}_j)$ for $j \geq i$. 
%\end{proof}

\begin{lemma}
\label{lemma-extension}
    Consider a cartesian square
    $$
    \begin{tikzcd}
        \mathcal{X}_U \ar[r, "f"] \ar[d] & \mathcal{X} \ar[d, "p"] \\
        U \ar[r, "g"] & S
    \end{tikzcd}
    $$
    where $g : U \to S$ be a faithfully flat morphism from an affine scheme to a scheme with an ample line bundle, and $\mathcal{X}$ is a quasi-compact and quasi-separated algebraic stack. Then given vector bundles $\mathcal{F}, \mathcal{G}$ on $\mathcal{X}$ and an extension
    \begin{equation}
    \label{equn-extension}
    0 \to f^*\mathcal{F} \to \mathcal{E} \to f^*\mathcal{G} \to 0
    \end{equation}
    on $\mathcal{X}_U$, there exists a vector bundle $\mathcal{P}$ on $S$, an extension
    $$
    0 \to \mathcal{F} \otimes p^*\mathcal{P} \to \mathcal{E}' \to \mathcal{G} \to 0
    $$
    of vector bundles on $\mathcal{X}$, and a split surjection $f^*\mathcal{E}' \to \mathcal{E}$ on $\mathcal{X}_U$. 
\end{lemma}

\begin{proof}
    We have
    $$
    \operatorname{Ext}^1_{\mathcal{X}_U}(f^*\mathcal{G}, f^*\mathcal{F}) = \operatorname{Ext}^1_{\mathcal{X}} (\mathcal{G}, \mathcal{F} \otimes f_* \mathcal{O}_{\mathcal{X}_U}) = \operatorname{colim}_i\operatorname{Ext}^1(\mathcal{G}, \mathcal{F} \otimes p^*\mathcal{P}_i).
    $$
    where $\mathcal{P}_i$ is a filtered system of vector bundles on $S$ whose colimit is $g_* \mathcal{O}_U$. Choose an index $i$ and an extension class
    \begin{equation}
    \label{equn-extensionapprox}
    0 \to \mathcal{F} \otimes p^*\mathcal{P}_i \to \mathcal{E}'_i \to \mathcal{G} \to 0
     \end{equation}
    which maps to the class of (\ref{equn-extension}) in the limit. Concretely, this means its pullback along $f$ fits into a commutative diagram
    \begin{equation}
    \label{equn-tworowdiagram}
    \begin{tikzcd}
        0 \ar[r] &f^*(\mathcal{F}\otimes p^*\mathcal{P}_i) \ar[r] \ar[d] & f^*\mathcal{E}'_i \ar[r] \ar[d] & f^*\mathcal{G} \ar[r] \ar[d] &0 \\
        0 \ar[r] &f^*\mathcal{F} \ar[r]  & \mathcal{E} \ar[r]  & f^*\mathcal{G} \ar[r] &0, \\
    \end{tikzcd}
    \end{equation}
    where the bottom extension is (\ref{equn-extension}) and is the pushout of the top extension along the map $f^*(\mathcal{F} \otimes p^*\mathcal{P}_i) \to f^*\mathcal{F}$ induced by the morphism $\mathcal{P}_i \to g_*\mathcal{O}_U$. We must show that after possibly increasing $i$, the map $f^*\mathcal{E}'_i \to \mathcal{E}$ is a split surjection. 

    In the colimit over $i$, the diagram (\ref{equn-tworowdiagram}) becomes identified with the diagram
    $$
 \begin{tikzcd}
        0 \ar[r] &f^*f_*f^*\mathcal{F} \ar[r] \ar[d] & f^*\mathcal{E}'
        \ar[r] \ar[d] & f^*\mathcal{G} \ar[r] \ar[d] &0 \\
        0 \ar[r] &f^*\mathcal{F} \ar[r]  & \mathcal{E} \ar[r]  & f^*\mathcal{G} \ar[r] &0, \\
    \end{tikzcd}
    $$
    in which the top extension is the pullback along $f$ of the extension of $\mathcal{G}$ by $f_*f^*\mathcal{F}$ corresponding to (\ref{equn-extension}) under $\operatorname{Ext}^1(\mathcal{G}, f_*f^*\mathcal{F}) = \operatorname{Ext}^1(f^*\mathcal{G}, f^*\mathcal{F})$; and the bottom extension is the pushout of the top extension along the map $f^*f_*f^*\mathcal{F} \to f^*\mathcal{F}$ induced by the counit of adjunction $f^*f_* \to 1$. It follows from abstract nonsense that $f^*\mathcal{E}' \to \mathcal{E}$ is a split surjection: Namely, it follows from the unit-counit identities that the top row can also be obtained as the pushout of the extension (\ref{equn-extension}) along the map $f^*\mathcal{F} \to f^*f_*f^*\mathcal{F}$ coming from the unit of adjunction, and this is an injection split by the map $f^*f_*f^*\mathcal{F} \to f^*\mathcal{F}$ induced by the counit. A splitting $\mathcal{E} \to f^*\mathcal{E}'$ factors through $f^*\mathcal{E}'_i$ for $i \gg 0$ and this proves the result.
   % 
   % 
   % 
   % By the previous Lemma, after possibly shrinking the index category, the system $f^*(\mathcal{F} \otimes p^*\mathcal{P}_i)$ is isomorphic to a system $f^*\mathcal{F} \oplus \mathcal{Q}_i$ where $\mathcal{Q}_i$ is a system of vector bundles on $\mathcal{X}_U$, and the left vertical arrow in the diagram is identified with projection onto the first factor. Then the class $\xi_i$ of the top extension lives in
  %  $$
   % \operatorname{Ext}^1(f^*\mathcal{G}, f^*\mathcal{F} \oplus \mathcal{Q}_i) = \operatorname{Ext}^1(f^*\mathcal{G}, f^*\mathcal{F}) \oplus \operatorname{Ext}^1(f^*\mathcal{G}, \mathcal{Q}_i),
  %  $$
   % so writing $\xi_i = (a, b_i)$, by our choice of extension class, for sufficiently large $i$ we have $b_i = 0$. But for such an $i$, the map $f^*\mathcal{E}' \to \mathcal{E}$ is a split surjection.
\end{proof}

\subsection{Proof of Theorem \ref{thm-faithfulbundleexists}}

We restate the result here fixing also the notation for the proof.

\begin{thm}
\label{thm-main}
    Let $S$ be a scheme with an ample line bundle. Let $G \to S$ be a finitely presented flat group scheme. Let $\mathcal{X} \to S$ be a gerbe banded by $G$. Assume there exist:
    \begin{enumerate}
        \item A finite locally free surjection $S' \to S$; and
        \item A faithful $G_{S'}$-equivariant vector bundle $\mathcal{E}$ on $S'$ which admits a filtration by $G_{S'}$-stable sub-bundles $\mathcal{E}^\bullet$ such that $\mathcal{E}^i/\mathcal{E}^{i+1}$ is a line bundle for each $i$;
    \end{enumerate}
    such that, if 
    $$\Psi^{\mathcal{E}^\bullet} : G_{S'} \to \prod_i \operatorname{Aut}(\mathcal{E}^i/\mathcal{E}^{i+1}) \cong \mathbf{G}_{m, S'}^n$$
    is as in (\ref{equn-maptotorus}), then the class of the pushforward gerbe $\Psi^{\mathcal{E}^{\bullet}}_*\mathcal{X}$ in
    $H^2(S'_{\acute{e}t}, \mathbf{G}_m^n)$ is torsion. Then $\mathcal{X}$ admits a faithful vector bundle and is a quotient stack.
\end{thm}

\begin{proof}%[Proof of Theorem \ref{thm-faithfulbundleexists}]
    If we find a faithful vector bundle on $\mathcal{X}_{S'}$ then we are done by Lemmas \ref{lemma-faithfulbundlemeaning} and \ref{lemma-resforcover}. We may therefore replace $S$ by $S'$. For each $i$, let $\mathcal{L}_i$ be the $G$-equivariant line bundle $\mathcal{E}^i/\mathcal{E}^{i+1}$ on $S$ and denote 
    $$
    \psi_i : G \to \operatorname{Aut}(\mathcal{E}^i/\mathcal{E}^{i+1}) \cong \mathbf{G}_{m, S}
    $$
    the associated character. Let $\mathcal{X} \to \mathcal{X}_i$ be the pushforward of $\mathcal{X}$ along $\psi_i$. 
    
    Then by assumption, $\mathcal{X}_i$ is a gerbe banded by $\mathbf{G}_m$ whose class in $H^2(S_{\acute{e}t}, \mathbf{G}_m)$ is torsion. By the Theorem of de Jong and Gabber \cite{dejong-gabber}, for each $i$ there exists a faithful 1-twisted vector bundle on each $\mathcal{X}_i$, and therefore by pullback a nowhere zero $\psi_i$-twisted vector bundle $\mathcal{F}_i$ on $\mathcal{X}$ for each $i$, see \ref{subsection-twistedsheaves}. We may assume each $\mathcal{F}_i$ has constant rank. Now choose an fppf covering $U \to S$ such that $\mathcal{X}_U \cong BG_U$ and $(\mathcal{F}_i)_U \cong (\mathcal{L}_i)_U^{\oplus r_i}$. %(see Lemma \ref{lemma-twistedlocalstr}, we have chosen $U$ so that the underlying line bundle of $(\mathcal{L}_i)_U$ is trivial). 

    We will prove by induction that for all $i$ there exists a vector bundle $\mathcal{G}_i$ on $\mathcal{X}$ and a split surjection $(\mathcal{G}_i)_U \twoheadrightarrow (\mathcal{E}_i)_U$. For $i = 1$, we may take $\mathcal{G}_1 = \mathcal{F}_1$ and this works. Now suppose we are given a vector bundle $\mathcal{G}_{i-1}$ and split surjection
    $(\mathcal{G}_{i-1})_U \to (\mathcal{E}_{i-1})_U$. We also have a split surjection $(\mathcal{F}_i)_U \to (\mathcal{L}_i)_U$. Now choose an extension that fits into a commutative diagram
    $$
    \begin{tikzcd}
        0 \ar[r] & (\mathcal{G}_{i-1})_U \ar[r] \ar[d] & \mathcal{H} \ar[r] \ar[d] & (\mathcal{F}_i)_U \ar[r] \ar[d] & 0 \\
        0 \ar[r] & (\mathcal{E}_{i-1})_U \ar[r] & (\mathcal{E}_{i})_U \ar[r] & (\mathcal{L}_i)_U \ar[r] & 0
    \end{tikzcd}
    $$
    in which the middle vertical arrow is a split surjection (we can do this because the functor $\operatorname{Ext}^1(-, -)$ is bi-additive). Now Lemma \ref{lemma-extension} gives us an extension
    $$
    0 \to \mathcal{G}_{i-1} \to \mathcal{G}_i \to \mathcal{F}_i \to 0
    $$
    such that there exists a split surjection $(\mathcal{G}_i)_U \to \mathcal{H}$ and therefore also a split surjection $(\mathcal{G}_i)_U \to (\mathcal{E}_i)_U$. This completes the induction.

    Now we are done because $\mathcal{G} = \mathcal{G}_n$ is faithful: This may be checked after pullback to $U$, where it surjects onto the faithful bundle $\mathcal{E}_U$. 
\end{proof}

\subsection{Remarks on the assumptions}
\label{subsection-remarks}

The assumption that $\mathcal{X}$ is torsion in Theorem \ref{thm-faithfulbundleexists} is somewhat inelegant because it depends on a choice of faithful bundle with a complete flag, as the following lemma shows.

For an integer $n > 0$ let $B_n \subset GL_n$ be the subgroup scheme of invertible upper triangular matrices. There is a homomorphism $d_n: B_n \to \mathbf{G}_m^n$ taking a matrix to its diagonal entries. 

\begin{lemma}
    There exist: a scheme $S$; a group scheme $G/S$; two different embeddings into upper triangular matrices, $\varphi: G \to B_{n, S}$ and $\psi : G \to B_{m, S}$ (both homomorphisms of $S$-group schemes which are closed immersions); and a gerbe $\mathcal{X}/S$ banded by $G$; such that: the class of the pushout gerbe $(d_n \circ \varphi)_*\mathcal{X}$ in $H^2(S_{\acute{e}t}, \mathbf{G}_m^n)$ is torsion while the class of $(d_m \circ \psi)_*\mathcal{X}$ in $H^2(S_{\acute{e}t}, \mathbf{G}_m^m)$ is not. 
\end{lemma}

\begin{proof}
    Let $S_0$ be any variety over the rational numbers with $H^2(S_0, \mathcal{O}_{S_0}) \neq 0$. Set $S = S[\epsilon] = S \otimes _{\mathbf{Q}} \mathbf{Q}[\epsilon]$. Let $G= \mathbf{G}_{a, S}$. Let $\mathcal{X}/S$ be any gerbe banded by $G$ such that the gerbe 
    $$
    \mathcal{X} \times _S S_0 \to S_0
    $$
    is not trivial; for example, the pullback of a non-trivial gerbe $\mathcal{X}_0 \to S_0$ banded by $\mathbf{G}_a$ along $S \to S_0$. Let $\varphi : G \to B_2$ be the usual embedding 
    $$
    f \mapsto
    \begin{pmatrix}
        1 & f \\
        0 & 1
    \end{pmatrix}
    $$
    and let $\psi : G \to B_3$ be the embedding
    $$
    f \mapsto
    \begin{pmatrix}
        1 & f & 0 \\
        0 & 1 & 0 \\
        0 & 0 & 1 + \epsilon f
    \end{pmatrix}.
    $$
    Then the composition $d_2 \circ \varphi : G \to \mathbf{G}_m^2$ is trivial so $(d_2 \circ \varphi)_* \mathcal{X}$ is certainly a torsion $\mathbf{G}_m^2$-gerbe -- it is trivial. On the other hand, we claim the gerbe $(d_3 \circ \psi)_*\mathcal{X}$ is not torsion. It suffices to show the pushout of $\mathcal{X}$ along the homomorphism $G \to \mathbf{G}_m$ given by $f \mapsto 1+\epsilon f$ is not torsion. %Since $\mathbf{G}_a$ and $\mathbf{G}_m$ are smooth, we have $H^2(S, \mathbf{G}_a) = H^2(S_{\acute{e}t}, \mathbf{G}_a)$ and similarly for $\mathbf{G}_m$.
    The homomorphism of sheaves $\mathbf{G}_a \to \mathbf{G}_m$ on $S_{\acute{e}t}$ given by $f \mapsto 1 + \epsilon f$ factors as
    $$
    \mathbf{G}_a \to i_*\mathbf{G}_{a, S_0} \to \mathbf{G}_m
    $$
    where $i : S_0 \to S$ is the inclusion. By assumption, the image of $[\mathcal{X}] \in H^2(S_{\acute{e}t}, \mathbf{G}_a)$ in $H^2(S_{\acute{e}t}, i_*\mathbf{G}_{a, S_0}) = H^2((S_0)_{\acute{e}t}, \mathbf{G}_a)$ is not zero, and therefore not torsion as $H^2((S_0)_{\acute{e}t}, \mathbf{G}_a) = H^2(S_0, \mathcal{O}_{S_0})$ is a $\mathbf{Q}$-vector space. We are therefore done if we can show:
    
    Claim: $H^2(S_0, \mathcal{O}_{S_0}) = H^2(S_{\acute{e}t}, i_*\mathbf{G}_{a, S_0}) \to H^2(S_{\acute{e}t}, \mathbf{G}_m)$ is injective.

    This comes from the long exact sequence of cohomology associated to the short exact sequence of sheaves
    $$
    1 \to i_*\mathbf{G}_{a, S_0} \xrightarrow{f \mapsto 1 + \epsilon f} \mathbf{G}_{m} \to i_*\mathbf{G}_{m, S_0} \to 1
    $$
    on $S_{\acute{e}t}$. The long exact sequence of cohomology gives 
    $$
    \cdots \to \operatorname{Pic}(S) \to \operatorname{Pic}(S_0) \to H^2(S_0, \mathcal{O}_{S_0}) \to H^2(S_{\acute{e}t}, \mathbf{G}_m) \to \cdots 
    $$
    but $\operatorname{Pic}(S) \to \operatorname{Pic}(S_0)$ is surjective as $i : S_0 \to S$ admits a retraction. By exactness, the claim follows. 
\end{proof}

In many interesting cases, the above problem does not occur.

\begin{lemma}
\label{lemma-independentofembedding}
    Let $S$ be a quasi-compact scheme. Let $G/S$ be a finitely presented, flat, trigonalizable group scheme. 
    Let $\mathcal{X} \to S$ be a gerbe banded by $G$. 
   If $G \to S$ is finite locally free or $S$ is regular, then $\mathcal{X}$ is torsion. In fact, for \emph{any} faithful $G$-equivariant vector bundle $\mathcal{E}$ with complete $G$-stable flag $\mathcal{E}^\bullet$ and associated homomorphism
   $$
   \Psi^{\mathcal{E}^\bullet} : G \to \mathbf{G}_{m, S}^n \cong \prod \operatorname{Aut} (\mathcal{E}^i/\mathcal{E}^{i+1}),
   $$
   the class of the pushout gerbe $\Psi^{\mathcal{E}^\bullet}_*\mathcal{X}$ in $H^2(S_{\acute{e}t}, \mathbf{G}_m^n)$ is torsion.
   \end{lemma}

\begin{proof}
    If $G\to S$ is finite locally free, then there is an integer $N > 0$ such that the morphism $G \xrightarrow{g\mapsto g^N}G$ is zero. Then the morphism $\Psi^{\mathcal{E}^\bullet}$ factors through $\mu_{N, S}^n \subset \mathbf{G}_{m, S}^n$ so the class of $\Psi^{\mathcal{E}^\bullet}_*\mathcal{X}$ is annihilated by $N$. 

    If $S$ is regular, then $H^2(S_{\acute{e}t}, \mathbf{G}_m^n)$ is torsion by \cite[Proposition 1.4]{GroupedeBrauerII}.
\end{proof}

Finally, it is natural to ask whether in Theorem \ref{thm-faithfulbundleexists} we can replace the assumption that $G \to S$ becomes trigonalizable after a finite locally free covering by the weaker assumption that it becomes trigonalizable after an fppf covering $S' \to S$. Indeed, when $G \to S$ is unipotent in the sense of \cite{bragg2025unipotentmorphisms}, %\cite{unipotentsubgroups} 
then this is possible by Theorem 5.5 of loc. cit. However, the following example shows that this is not possible in general when $G / S$ is a torus.

\begin{example}
    Let $S$ be a quasi-compact scheme. Let $G \to S$ be a torus. Then by \cite[Expos\'e XI, Remarque 4.6]{SGA3_II} the stack $BG_S$ admits a faithful vector bundle if and only if $G \to S$ is iso-trivial, i.e., there exists a finite \'etale covering $S' \to S$ such that $G_{S'} \cong \mathbf{G}_{m, S'}^n$. When $S$ is a nodal curve, there exists a torus $T \to S$ of relative dimension 2 which is not isotrivial \cite[Expos\'e X, 1.6]{SGA3_II} and therefore $BT_S$ does not satisfy the resolution property. However, $T \to S$ becomes trigonalizable after an \'etale covering, but not after a finite \'etale covering.
\end{example}

\section{Examples of groups satisfying the hypotheses}\label{S:section4}

Many group schemes satisfy the hypotheses of Theorem \ref{thm-faithfulbundleexists}. The following well-known result describes the trigonalizable group schemes over a field. See for example {\cite[Chapter 16, Section a]{MilneAlgGroups}}.

\begin{lemma}
\label{lemma-trigoverfield}
    Let $k$ be a field. Let $G$ be an affine group scheme of finite type over $k$. The following are equivalent:
    \begin{enumerate}
        \item $G$ is trigonalizable over $k$.
        \item $G$ is an extension of a diagonalizable group scheme by a unipotent group scheme, both of finite type over $k$. 
        \item Every finite-dimensional representation of $G$ admits a complete flag of $G$-stable subspaces. 
    \end{enumerate}
\end{lemma}

\qed 

Many group schemes over a field become trigonalizable %become trigonalizable 
after passage to the algebraic closure.

\begin{lemma}
\label{lemma-liekolchin}
    Let $k$ be a field with algebraic closure $\bar{k}$. Let $G$ be an affine group scheme of finite type over $k$. If either: (a) $G$ is commutative, or (b) $G$ is smooth, solvable, and connected, then $G_{\bar{k}}$ is a trigonalizable group scheme over $\bar{k}$.
\end{lemma}

\begin{proof}
If $G$ is commutative, then $G_{\bar{k}}$ is a product of a unipotent group scheme and a diagonalizable group scheme \cite[Expos\'e XVII, 7.2.1]{SGA3_II},  so by Lemma \ref{lemma-trigoverfield} we see that $G_{\bar{k}}$ is trigonalizable. 

If $G$ is smooth, solvable, and connected then $G_{\bar{k}}$ is trigonalizable by the Lie--Kolchin Theorem, see \cite[Theorem 16.30]{MilneAlgGroups}.
\end{proof}

Over a Dedekind scheme, a flat finite type affine group scheme is trigonalizable if and only if its generic fiber is.

\begin{prop}
\label{prop-genericallytrig}
    Assume $S$ is a regular  integral scheme of dimension 1 with function field $K$. Assume $G \to S$ is a flat, relatively affine group scheme of finite type. Assume the $K$-group scheme $G_K$ is trigonalizable. Then so is the $S$-group scheme $G$.
\end{prop}

\begin{proof}
See \cite[Proof of Lemma 3.1]{unipotentsubgroups}, which is essentially the same. By \cite[VIB.13.2]{SGA3_1}, $G$ admits a faithful vector bundle $\mathcal{E}$. By the assumption that $G_K$ is trigonalizable and Lemma \ref{lemma-trigoverfield}, the $G_K$-representation $\mathcal{E}_K$ admits a complete flag $W^\bullet$ of $G$-stable subspaces. For each $G_K$-stable subspace $W^i \subset \mathcal{E}_K$, the $G$-submodule $\mathcal{E}^i = \mathcal{E} \cap \eta_*W^i \subset \mathcal{E}$ (here $\eta : \operatorname{Spec}(K) \to S$ is the inclusion) is a $G$-equivariant sub-bundle of $\mathcal{E}$ whose stalk at the generic point is $W^i \subset \mathcal{E}_K$. Thus $\mathcal{E}^\bullet$ is a complete flag of $G$-stable sub-bundles and we see that $G$ is trigonalizable. 
\end{proof}

\begin{corollary}
\label{corollary-geometricallytrig}
    Assume $S$ is a regular, Nagata, integral scheme of dimension 1 with function field $K$. Assume $G \to S$ is a flat, relatively affine group scheme of finite type. Let $\bar{K}$ be an algebraic closure of $K$ and assume the $\bar{K}$-group scheme $G_{\bar{K}}$ is trigonalizable. Then there is a finite locally free surjective morphism $T \to S$ such that $G_T$ is trigonalizable over $T$. 
\end{corollary}

\begin{proof}
    The hypotheses give a faithful upper triangular representation $G_{\bar{K}} \to GL_{N, \bar{K}}$ and hence a finite extension $L/K$ and a faithful upper triangular representation $G_{L} \to GL_{n, L}$. Let $T$ be the normalization of $S$ in $L$. Then $T \to S$ is finite locally free since $S$ is Nagata and regular of dimension $1$, and $G_T \to T$ satisfies the hypotheses of Proposition \ref{prop-genericallytrig} so we see that $G_T$ is trigonalizable as a group scheme over $T$. 
\end{proof}

Combining Corollary \ref{corollary-geometricallytrig} with the following lemma, whose proof is immediate,  gives many examples of locally trigonalizable group schemes over a general base.

\begin{lemma}
    Let $T \to S$ be a morphism of schemes. Let $G \to S$ be a relatively affine finitely presented flat group scheme. Assume $G$ is trigonalizable (resp. locally trigonalizable). Then so is the base change $G_T$. \qed
\end{lemma}

\begin{example}
Let $A$ be a finite abelian group and $S$ a scheme. Denote $A_S \to S$ the constant group scheme given by $A$. Then $A$ is locally trigonalizable since it is the base change of the constant group scheme $A_{\operatorname{Spec}(\mathbf{Z})} \to \operatorname{Spec}(\mathbf{Z})$, and the latter is locally trigonalizable by Corollary \ref{corollary-geometricallytrig} and Lemma \ref{lemma-liekolchin}.
\end{example}

%We do not currently know an example of a finite locally free commutative group scheme $G \to S$ which does not become trigonalizable after base change along some finite locally free surjective morphism $T \to S$.

Slightly more generally we have the following.

\begin{prop}
\label{prop-etalegpschemesareloctrig}
Let $S$ be a scheme and $G \to S$ a finite \'etale commutative group scheme. Then there exists a finite locally free surjective morphism $T \to S$ such that $G_T \to T$ is trigonalizable. 
\end{prop}

\begin{proof}
    We may assume all the geometric fibers $G_{\bar{s}}$ are isomorphic as abstract abelian groups (in general $S$ is a finite disjoint union of open and closed subsets over which this is true). Then there is a finite \'etale surjective morphism $S' \to S$ such that the group scheme $G_{S'} \to S'$ is a constant abelian group. Then by the preceding discussion, $S'$ has a finite locally free surjective cover $T \to S'$ such that $G_T$ is trigonalizable over $T$. 
\end{proof}

This proves a result we promised in the introduction.

\begin{corollary}
\label{corollary-finiteetalecommutative}
    Let $S$ be a scheme with an ample line bundle. Let $G \to S$ be a finite \'etale commutative group scheme. Let $\mathcal{X} \to S$ be a gerbe banded by $G$. Then $\mathcal{X}$ admits a faithful vector bundle, is a quotient stack, and satisfies the resolution property.
\end{corollary}

\begin{proof}
    The group scheme $G/S$ is locally trigonalizable by Proposition \ref{prop-etalegpschemesareloctrig}, and the gerbe $\mathcal{X} \to S$ is torsion by Lemma \ref{lemma-independentofembedding}. Therefore, the conclusion follows from Theorem \ref{thm-faithfulbundleexists} and Theorem \ref{thm-finiteflatcover}.
\end{proof}

\section{Smooth DM stacks with quasi-projective coarse space}\label{S:section5}

In this section, we prove Theorem \ref{thm-smoothDM}. The proof is exactly the same as \cite[Proof of Theorem 2.2]{kreschvistoli} except that it uses Corollary \ref{corollary-finiteetalecommutative} instead of Gabber's Theorem \cite{dejong-gabber}. 

\begin{lemma}
\label{lemma-etalegerbe}
Let $S$ be a scheme with an ample line bundle. Let $\mathcal{X} \to S$ be a separated \'etale gerbe. Then $\mathcal{X}$ admits a faithful vector bundle, is a quotient stack, and satisfies the resolution property.
\end{lemma}

\begin{proof}
    We may assume that there is a finite abstract group $G$ such that for every $s \in S$, the geometric fiber $\mathcal{X}_{\bar{s}}$ is isomorphic to $BG_{\bar{s}}$. In general, $S$ is a finite disoint union of open and closed subschemes $S_i$ for which this is true and we may prove the result for each $S_i$ individually. 

    In this case, $\mathcal{X} \to S$ is a $G$-gerbe, that is, there is an fppf cover $U \to S$ such that $\mathcal{X}_U \cong BG_{U}$ as stacks over $U$. By \cite[Proposition 3.5]{ehkv}, there is a finite \'etale covering $S' \to S$ such that $\mathcal{X}_{S'} \to S'$ is a gerbe banded by $G$. For $S' \to S$ we may take the $\operatorname{Out}(G)$-torsor which is the isomorphism sheaf between the band of $\mathcal{X}$ and the band associated to $G$ in the stack of bands on $(Sch/S)_{fppf}$, see \cite[IV.1.1]{Giraud}. It suffices to prove the result for $\mathcal{X}_{S'}$ by Lemma \ref{lemma-resforcover}. Thus we may assume $\mathcal{X} \to S$ is banded by $G$. Then by \cite[Section 1.2.5]{bandedbycenter} or \cite[IV.3.3]{Giraud}, there is a finite \'etale surjective morphism $\mathcal{Y} \to \mathcal{X}$ such that the composition $\mathcal{Y} \to S$ is a gerbe banded by the center of $G$. Namely, $\mathcal{Y}$ is the stack of equivalences from $BG_{S}$ to $\mathcal{X}$ over $S$ which induce the identity on bands. Again, Lemma \ref{lemma-resforcover} shows we may replace $\mathcal{X}$ with $\mathcal{Y}$ and therefore assume $G$ is commutative. Now the result follows from Corollary \ref{corollary-finiteetalecommutative}. 
\end{proof}

\begin{proof}[Proof of Theorem \ref{thm-smoothDM}]
The assumptions imply that $\mathcal{X}$ is normal, so by \cite[Proposition 2.1]{HomStacks}, the coarse space morphism $\mathcal{X} \to X$ factors as
$\mathcal{X} \to \mathcal{X}' \to X$ where $\mathcal{X} \to \mathcal{X}'$ is a separated \'etale gerbe and $\mathcal{X}' \to X$, which is the coarse space morphism of $\mathcal{X}'$, is an isomorphism over a dense open of $X$. By \cite[Theorem 2.18]{ehkv}, the stack $\mathcal{X}'$ is a quotient stack, and so by Theorem \ref{thm-finiteflatcover}, there is a finite locally free surjective morphism $Y \to \mathcal{X}'$ where $Y$ is a scheme with an ample line bundle. Then the projection 
$\mathcal{X} \times_{\mathcal{X}'} Y \to Y$ is a separated \'etale gerbe, so by Lemma \ref{lemma-etalegerbe}, the stack $\mathcal{X} \times_{\mathcal{X}'} Y$ admits a faithful vector bundle. Since this stack admits a finite locally free surjective morphism to $\mathcal{X}$ we conclude by Lemma \ref{lemma-resforcover}.
\end{proof}

\bibliographystyle{amsplain}
\bibliography{references}

\textsc{UC Berkeley Department of Mathematics, Department of Mathematics
970 Evans Hall, MC 3840
Berkeley, CA 94720-3840}

Email: \href{nolander@berkeley.edu}{nolander@berkeley.edu}

\textsc{UC Berkeley Department of Mathematics, Department of Mathematics
970 Evans Hall, MC 3840
Berkeley, CA 94720-3840}

Email: \href{martinolsson@berkeley.edu}{martinolsson@berkeley.edu}

\end{document}